\documentclass[12pt]{article}
\usepackage[margin=1in]{geometry}
\usepackage{amsmath,amssymb,amsthm}
\usepackage[dvips]{graphicx}
\usepackage[mathscr]{euscript}
\usepackage{color}
\usepackage{verbatim}
\usepackage{epsfig}

\newtheorem{theorem}{Theorem}[section]
\newtheorem{lemma}[theorem]{Lemma}

\theoremstyle{definition}
\newtheorem{definition}[theorem]{Definition}

\newtheorem{examples}[theorem]{Examples}

\theoremstyle{remark}
\newtheorem*{remark}{Remark}

\numberwithin{equation}{section}

\newcommand{\R}{\mathbb R}

\newcommand{\Vol}{\mbox{Vol}}

\begin{document}

\title{The Archimedean Projection Property}
\author{Vincent Coll, Jeff Dodd, and Michael Harrison}

\maketitle

\begin{abstract}
\noindent Let $H$ be a hypersurface in $\R^n$ and let $\pi$ be an orthogonal projection in $\R^n$ restricted to $H$.  We say that $H$ satisfies the $\emph{Archimedean}$ $\emph{projection}$ $\emph{property}$ corresponding to $\pi$ if there exists a constant $C$ such that $\Vol(\pi^{-1}(U)) = C \cdot \Vol(U)$ for every measurable $U$ in the range of $\pi$.  It is well-known that the $(n-1)$-dimensional sphere, as a hypersurface in $\R^n$, satisfies the Archimedean projection property corresponding to any codimension 2 orthogonal projection in $\R^n$, the range of any such  projection being an $(n-2)$-dimensional ball.  Here we construct new hypersurfaces that satisfy Archimedean projection properties.  Our construction works for any projection codimension $k$, $2 \leq k \leq n - 1$, and it allows us to specify a wide variety of desired projection ranges $\Omega^{n-k} \subset \R^{n-k}$.  Letting $\Omega^{n-k}$ be an $(n-k)$-dimensional ball for each $k$, it produces a new family of smooth, compact hypersurfaces in $\R^n$ satisfying codimension $k$ Archimedean projection properties that includes, in the special case $k = 2$, the $(n-1)$-dimensional spheres.
\end{abstract}


\bigskip


\bigskip

\medskip


\noindent \section{Introduction}
\noindent  Archimedes' theorem states that if two parallel planes slice through a $2$-sphere, the surface area of the resulting zone is proportional to the distance between the planes.  This property of the $2$-sphere can be recast as a projection property that generalizes to higher dimensional spheres:  consider a unit $(n-1)$-sphere $S^{n-1}(1) \subset \R^n$, and let $\pi: S^{n-1}(1) \rightarrow \R^{n-2}$ be a codimension 2 orthogonal projection in $\R^n$ restricted to $S^{n-1}(1)$, so that the range of $\pi$ is a unit $(n-2)$-dimensional ball $B^{n-2}(1)$.  Then for every measurable $U \subset B^{n-2}(1)$, $\Vol(\pi^{-1}(U)) = 2\pi \cdot \Vol(U)$.  This projection property of the sphere has appeared in a number of different mathematical contexts.  For example, it was recently employed by K.\ Bezdek and R.\ Connelly \cite{Bezdek} to resolve an important special case of a longstanding conjecture of Kneser and Poulsen having to do with disk coverings in the plane. And, in probability theory, it is essentially the statement that the canonical projection map from $S^{n-1}$ onto the the $(n - 2)$-ball $B^{n-2}$ is measure preserving, so that a uniform measure on $S^{n-1}$ induces a uniform measure on $B^{n-2}$.

To the best of our knowledge, this projection property of the spheres has not been given a name, so we introduce the following terminology.
\begin{definition}
Suppose that $H$ is a hypersurface embedded in $\R^n = \R^{n-k} \times \R^k$, and $\pi_{n-k}$ is the codimension $n-k$ orthogonal projection from $\R^n$ onto $\R^{n - k}$ restricted to $H$.  We say that $H$ satisfies the \emph{Archimedean projection property (``APP'')} corresponding to $\pi_{n-k}$ if there exists a non-zero \emph{proportionality constant} $C$ such that for every measurable set $U \subset \pi_{n-k}(H)$,
\begin{eqnarray}\label{APP}
\Vol(\pi_{n-k}^{-1}(U)) = C \cdot \Vol(U).
\end{eqnarray}
\end{definition}
We know of only two classes of hypersurfaces in $\R^n$ that non-trivially satisfy an Archimedan projection property:  the $(n - 1)$-spheres, with projection codimension 2, and hypersurfaces of revolution called equizonal ovaloids, with projection codimension $n-1$ (see \cite{CollDoddHarrison}, \cite{CollHarrison}, and \cite{CollandDodd}).  Here we present a new method for constructing hypersurfaces satisfying Archimedean projection properties.  Our method is flexible enough that, given any projection codimension $k \geq 2$, and any of a wide variety of projection ranges, we can produce a hypersurface non-trivially satisfying the corresponding Archimedean projection property.

Our constructions are carried out within a class of hypersurfaces that we call \emph{spherical} \emph{arrays} \cite{CollHarrison2}.
\begin{definition}
A {\em spherical array} is a warped product of the form
$$ H = \Omega^{n-k}\times_f S^{k-1} $$
embedded in $\R^n = \R^{n-k} \times \R^k$.  The region $\Omega^{n-k} \subset \R^{n-k}$ is the \emph{base} of $H$ and $f:\Omega^{n-k} \to [0,\infty)$ is a \emph{warping} \emph{function} which specifies the radius of a $(k - 1)$-dimensional spherical fiber centered at each point in $\Omega^{n-k}$.
\end{definition}
\noindent We begin our investigation by describing the possible warping functions of a spherical array that satisfies the Archimedean projection property corresponding to orthogonal projection onto its base.  Among these spherical arrays, we identify those that are closed, bound a strictly convex interior, and are of at least $C^2$ smoothness.  We call these hypersurfaces {\em Archimedean spherical arrays}.  We find that for each ambient space dimension $n\geq 3$ and projection codimension $2 \leq k \leq n - 1$ there is (up to scaling) a unique Archimedean spherical array, and that the Archimedean spherical arrays are smoother than we initially require:  when $k$ is even, they are analytic, and when $k$ is odd, they are of $C^{k-1}$ smoothness.  The family of Archimedean spherical arrays includes, and interpolates naturally between, the spheres and the equizonal ovaloids.

The paper is organized as follows:  Section 2 is a presentation of our main theorems.  Section 3 is devoted to proofs of the main theorems; it features an explicit formula for the volume element on an arbitrary spherical array that is of independent interest.  Section 4 contains formulas for the $(n-1)$-volumes of the Archimedean spherical arrays.  Concluding comments and questions for further study of the Archimedean projection property can be found in Section 5.

\noindent \section{Main Results}
When discussing a spherical array $H$, we will always use the coordinate system indicated in Figure 1:
\begin{align*}
H & = \left\{ (x^{\prime \prime},x^{\prime}) \in \Omega^{n-k} \times \R^k : \|x^{\prime}\|^2 = [f(x^{\prime \prime})]^2 \right\} \nonumber \\
& = \left\{ (x_1, \dots ,x_n) \in \R^n : x_1^2+x_2^2 + \dots + x_k^2 = [f(x_{k+1}, \dots, x_n)]^2 \right\}.
\end{align*}

When $k = n - 1$, $H$ is an embedded hypersurface of revolution:  an $(n - 1)$-dimensional hypersurface embedded in $\R^n$ whose cross-sections orthogonal to the $x_n$-axis are $(n-2)$-dimensional spheres.
\begin{figure}[ht!]
\centerline{
\includegraphics[height=3.5in]{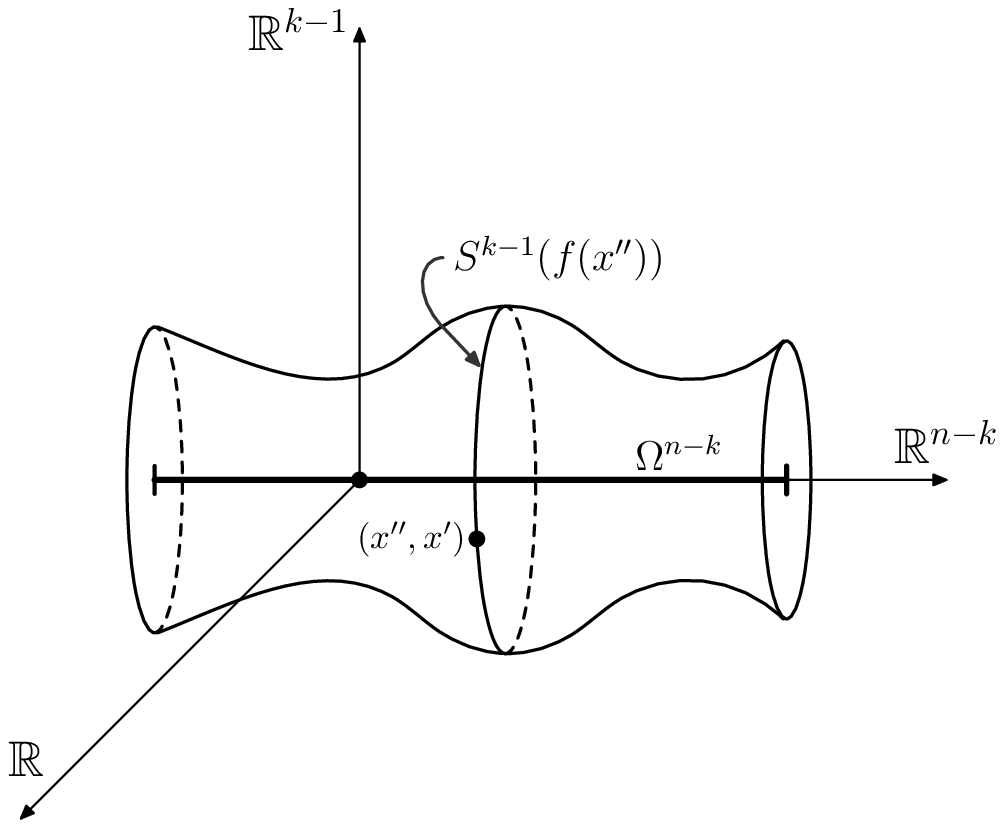}
}
\caption{A spherical array:  $x^{\prime \prime} = (x_{k + 1}, \ldots, x_n) \in \R^{n-k}$, $x^{\prime} = (x_1, \ldots, x_k) \in \R \times \R^{k-1} = \R^k$.}
\label{fig:Area}
\end{figure}

Given any orthogonal projection $\pi_{n-k}:  \R^n \rightarrow \R^{n-k}$ and any projection range $\Omega^{n-k} \subset R^{n-k}$, the spherical array $H$, with base $\Omega^{n-k}$ and with the constant warping function $f = 1$ (which would look like a cylinder in Figure 1) trivially satisfies the Archimedean projection property corresponding to the orthogonal projection $\pi_{n-k}$ onto its base.  In this case, the proportionality constant $C$ in (\ref{APP}) is the volume of the unit $(k-1)$-dimensional sphere, $C = \Vol(S^{k-1}(1))$.  If a spherical array with a non-constant warping function satisfies the same Archimedean projection property over the same base $\Omega^{n-k}$ and with the same proportionality constant $C = \Vol(S^{k-1}(1))$, then its non-constant warping function $f(x^{\prime \prime})$ must perform a balancing act.  Namely, suppose that as $x^{\prime \prime}$ moves in some direction in $\Omega^{n-k}$, the radius $f(x^{\prime \prime})$ of the corresponding spherical fibers increases.  Locally, near $x^{\prime \prime}$, this increase in the radius $f(x^{\prime \prime})$ tends to increase the ratio of the volume of the hypersurface to the volume of its projection onto the base.  The only way that this effect can be offset is if the rate of increase of the radius $f(x^{\prime \prime})$ decreases, since this has the opposite effect, tending to decrease the volume of the projection of the hypersurface onto the base relative to the volume of the hypersurface itself.

In the special case of the ordinary unit $2$-sphere $S^2(1)$, the balance between these opposing effects is easy to visualize:  the base of this spherical array is a diameter of the sphere, and as $x^{\prime \prime}$ moves away from an endpoint of this diameter, the radius of the corresponding circular cross-sections centered at $x^{\prime \prime}$ increases, but at a decreasing rate.  When the radius of the circular cross-sections reaches 1, the rate of increase of the radius reaches 0, and the radius cannot increase further.  For a spherical array in general, due to the multiple directions in which $x^{\prime \prime}$ can move in the base $\Omega^{n-k}$, an analogous balance between size and rate of growth of the spherical fibers is much more complicated to imagine geometrically.  But we have found that it is surprisingly simple to arrange analytically.  In particular, the key principle that emerges in our constructions is the remarkable fact that the warping function $f$ that maintains this balance is always given locally by the composition of two functions.

The first function, which plays the role of distance from a pole along a diameter in the case of the unit $2$-sphere, is an appropriate solution $\omega$ of the eikonal equation $\| \nabla \omega(x^{\prime \prime}) \| = 1$ on the base $\Omega^{n - k}$.  The second function, which determines  how the radii of the spherical fibers respond to the value of $\omega(x^{\prime \prime})$ as $x^{\prime \prime}$ changes in $\Omega^{n-k}$, is an increasing, concave down scalar function $f_k: [0, M_k] \rightarrow [0,1]$ that depends only on the projection codimension $k$.  When $k = 2$, $M_k = 1$ and the graph of $f_k$ is a quarter circle.  As $k$ increases, $M_k$ decreases, but the range of $f_k$ is always $[0,1]$, representing the possible radii of the spherical fibers (see Figure \ref{fig:profile}).  We call $f_k$ the {\em codimension $k$ Archimedean scaling function}.

Our first theorem describes precisely how these two functions, $f_k$ and $\omega$, work in tandem to impose an Archimedean projection property on a spherical array.  An analytical description of the Archimedean scaling functions $f_k$ arises naturally in the proof of the theorem in Section 2, and formulas for the corresponding numbers $M_k$ are derived in Section 4.

\begin{figure}[h]
\centerline{
\includegraphics[width=3.5in]{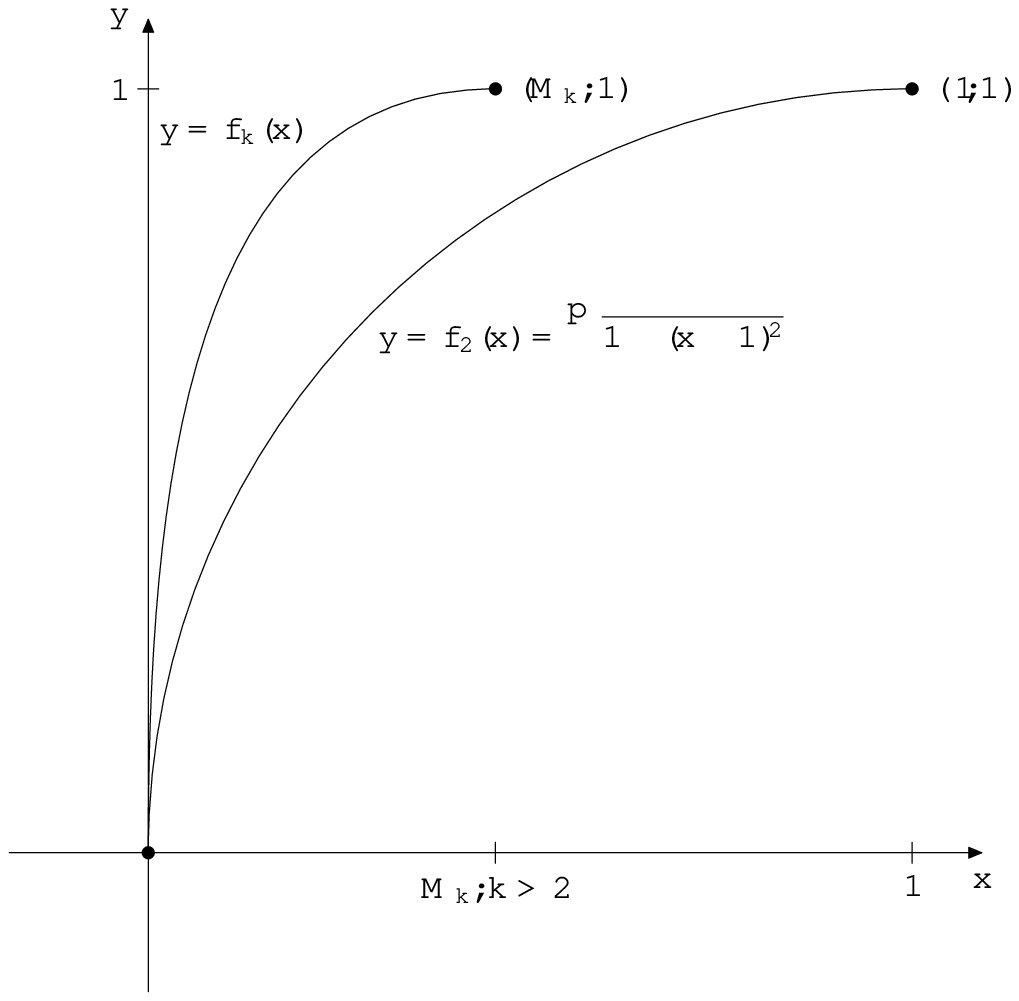}
}
\caption{The Archimedean scaling functions $f_k$ for $k = 2$ and $k > 2$.}
\label{fig:profile}
 \end{figure}

\begin{theorem}\label{Main theorem 1}
Consider the spherical array $H = \Omega^{n-k} \times_f S^{k-1} \subset \R^{n-k} \times \R^k$, where $\Omega^{n-k}$ is a domain in $\R^{n-k}$.
\begin{itemize}
\item[(a)] Suppose $f$ is a nonnegative, continuously differentiable, function on $\Omega^{n - k}$ such that $H$ satisfies the APP corresponding to orthogonal projection onto its base with the constant of proportionality in (\ref{APP}) being $C = \Vol(S^{k-1}(1))$.  Then $f(x^{\prime \prime}) \leq 1$ for all $x^{\prime \prime} \in \Omega^{n-k}$, and on $A = \{ x^{\prime \prime} \in \Omega^{n-k} : 0 < f(x^{\prime \prime}) < 1\}$, $f = f_k\circ \omega$, where $f_k$ is the codimension $k$ Archimedean scaling function and $\omega$ is a positive, continuously differentiable solution of the eikonal equation $\| \nabla \omega(x^{\prime \prime}) \| = 1$ on $A$.
\item[(b)] Conversely, suppose that $\omega$ is a positive, continuously differentiable, solution of the eikonal equation  $\| \nabla \omega(x^{\prime \prime}) \|= 1$ on $\Omega^{n-k}$, such that $\omega(x^{\prime \prime}) \leq M_k$, for all $x^{\prime \prime} \in \Omega^{n-k}$.  Then for $B = \{ x^{\prime \prime} \in \Omega^{n-k} : 0 < \omega(x^{\prime \prime}) \leq M_k \}$, $f = f_k \circ \omega$ is a continuously differentiable function such that the spherical array $B \times_f S^{k-1} \subset \R^{n-k} \times \R^k$ satisfies the APP corresponding to orthogonal projection onto its base with the constant of proportionality in (\ref{APP}) being $C = \Vol(S^{k-1}(1))$.
\end{itemize}
\end{theorem}
Theorem \ref{Main theorem 1} allows for the construction of many different spherical arrays satisfying the APPs corresponding to orthogonal projection onto their bases.
\begin{examples}
Suppose the domain $\Omega^{n-k} \subset \R^{n-k}$ is bounded, and that the maximum distance from any point $x^{\prime \prime}$ in $\Omega^{n-k}$ to the boundary of $\Omega^{n-k}$ is exactly $M_k$.  Let $f = f_k \circ \omega$, where $\omega$ is the distance to the boundary function on $\Omega^{n-k}$.  According to Y.\ Li and L.\ Nirenberg \cite{Li}, if the boundary of $\Omega^{n-k}$ is of at least $C^{2,1}$ smoothness, then $\omega$ is continuously differentiable and solves the eikonal equation $\| \nabla \omega(x^{\prime \prime}) \|= 1$ on $\Omega^{n-k} \backslash \Sigma$, where $\Sigma$ is a closed, path connected set whose (n-k-1)-dimensional Hausdorff measure is finite, and therefore whose (n-k)-dimensional Lebesgue measure is zero.  So by part (b) of Theorem \ref{Main theorem 1}, the spherical array $H = \Omega^{n-k} \times_f S^{k-1}$ is a closed and bounded hypersurface in $\R^n$ that is of $C^1$ smoothness outside of a set of measure zero and satisfies the APP corresponding to orthogonal projection onto its base $\Omega^{n-k}$.  \end{examples}

Among the many spherical arrays satisfying an APP corresponding to orthogonal projection onto the base, we can identify those that resemble spheres as closely as possible in the following sense:
\begin{definition} \label{ASA}
Suppose that a spherical array $H = \Omega^{n-k} \times_f S^{k-1} \subset \R^{n-k} \times \R^k$ satisfies the APP corresponding to the codimension $k$ orthogonal projection $\pi_{n-k}:  \R^n \rightarrow \R^{n - k}$.  If, in addition, $H$ is closed, bounds a strictly convex interior, and is of at least $C^2$ smoothness, then we refer to $H$ as a \emph{$k$-Archimedean spherical array}, or simply as an \emph{Archimedean spherical array} when $k$ is understood.
\end{definition}
There is only one way to build an Archimedean spherical array:
\begin{theorem}\label{Main theorem 2}
For each ambient space dimension $n \geq 3$ and for each projection codimension $k \geq 2$, there is (up to scaling) a unique $(n-1)$-dimensional $k$-Archimedean spherical array embedded in $\R^n$ that we denote $\mathscr{A}_k^{n-1}$.  When the constant of proportionality in (\ref{APP}) is $C = \Vol(S^{k-1}(1))$, these are the spherical arrays for which the base $\Omega^{n - k}$ is an $(n - k)$-dimensional ball of radius $M_k$,  $B^{n-k}(M_k)$, and the warping function is $f = f_k \circ \omega$, where $\omega$ is the distance to the boundary function associated with this ball and $f_k$ is the codimension $k$ Archimedean scaling function.
\end{theorem}
While an Archimedean spherical array has at least $C^2$ smoothness by definition, the Archimedean spherical arrays turn out to be smoother than this for all projection codimensions $k \neq 3$:
\begin{theorem}\label{Main theorem 3}
The Archimedean spherical array $\mathscr{A}_k^{n-1}$ is real analytic when $k$ is even and of $C^{k-1}$ smoothness when $k$ is odd.
\end{theorem}
\begin{examples}
When the projection codimension $k$ is 2, the Archimedean spherical arrays are spheres:  $\mathscr{A}_{2}^{2}$ is the $2$-sphere and, more generally, $\mathscr{A}_{2}^{n-1}$ is the $(n-1)$-sphere.  At the other extreme, when the projection codimension $k$ is $n - 1$, the Archimedean spherical arrays $\mathscr{A}_{n-1}^{n-1}$ are the equizonal ovaloids:  embedded hypersurfaces of revolution which were developed and studied in detail in \cite{CollDoddHarrison} and \cite{CollandDodd}.
\end{examples}
\begin{remark}
We maintain a proportionality constant of $C = \Vol(S^{k-1}(1))$ throughout our discussion since a hypersurface satisfying an APP with $C \neq \Vol(S^{k-1}(1))$ is just a rescaled version of a hypersurface satisfying an APP with $C = \Vol(S^{k-1}(1))$.
\end{remark}

\section{Proofs of Theorems \ref{Main theorem 1}, \ref{Main theorem 2}, and \ref{Main theorem 3}}
\noindent In order to prove Theorem \ref{Main theorem 1}, which characterizes the spherical arrays $H=\Omega^{n-k} \times_f S^{k-1}$ satisfying the APP corresponding to orthogonal projection onto the base $\Omega^{n - k}$, we require a formula for the $(n - 1)$-volume of the portion of $H$ lying over a measurable $U \subset\Omega^{n-k}$.  This formula results from the following sequence of three lemmas.
\begin{lemma}
\label{lem:hypersurfacegraph}
Given an open set $U \subset \R^{k-1}$ and a continuously differentiable function $h:U \to \R$, the $(k-1)$-volume of the graph of the function $x_1 = h(x_2, \dots, x_k)$ is given by
$$\int_U \sqrt{1 + h_{x_2}^2 + \dots + h_{x_k}^2} \mbox{ d}x_2 \dots \mbox{d}x_k.$$
\end{lemma}
\begin{proof} To compute the volume form, we require the quantity $\sqrt{\det(g)}$, where $g$ is the matrix representing the first fundamental form.  We first compute the matrix of partial derivatives
$$ B= \left( \begin{array}{cccc}
h_{x_2} & h_{x_3} & \dots & h_{x_k} \\
1 & 0 & \dots & 0 \\
0 & 1 & \dots & 0 \\
\vdots & \vdots & \ddots & \vdots \\
0 & 0 & \dots & 1
\end{array} \right),$$
and then $g$ is given by
$$g = B^TB = \left( \begin{array}{cccc}
1+h_{x_2}^2 & h_{x_2} h_{x_3} & \dots & h_{x_2} h_{x_k} \\
h_{x_3} h_{x_2} & 1+h_{x_3}^2 & \dots & h_{x_3} h_{x_k} \\
\vdots & \vdots & \ddots & \vdots \\
h_{x_k} h_{x_2} & h_{x_k} h_{x_3} & \dots & 1 + h_{x_k}^2
\end{array} \right).$$
We rewrite $g$ in the special form,
$g = I_{k-1} + {\bf h}^T {\bf h}$, where \[  {\bf h} = \left( \begin{array}{cccc}
h_{x_2} & h_{x_3} & \cdots & h_{x_k} \end{array} \right). \]
From this, we see that any vector orthogonal to ${\bf h}$ is an eigenvector of $g$ with eigenvalue $1$; there are $k-2$ of these.  The other eigenvector is ${\bf h}$ itself, with eigenvalue $1 + h_{x_2}^2 + \dots + h_{x_k}^2$, and so $\sqrt{\det(g)} = \sqrt{1 + h_{x_2}^2 + \dots + h_{x_k}^2}.$ \hfill
\end{proof}
\begin{lemma}
\label{lem:spherearea}
The volume of a $(k-1)$-dimensional sphere of radius $r$ is given by
\begin{align}
\label{eqn:spherearea}
\Vol(S^{k-1}(r)) = 2\int_{B^{k-1}(r)} \frac{r}{\sqrt{r^2 - x_2^2 - \dots - x_k^2}} \mbox{ d} x_2 \dots \mbox{d} x_k.
\end{align}
\end{lemma}
\begin{proof} A hemisphere can be represented by the graph of the function $x_1 = h(x_2,\dots,x_k) = \sqrt{r^2-x_2^2-\dots - x_k^2}$, over the domain $B^{k-1}(r)$. We compute
\begin{align*}
1 + h_{x_2}^2 + \dots + h_{x_n}^2 = 1 + \frac{x_2^2 + \dots + x_k^2}{r^2 - x_2^2 - \dots - x_k^2} = \frac{r^2}{r^2 - x_2^2 - \dots - x_k^2}.
\end{align*}
The result then follows from Lemma \ref{lem:hypersurfacegraph}.
\end{proof}
\begin{lemma} The $(n-1)$-volume of the spherical array $H=\Omega^{n-k} \times_f S^{k-1}$ lying over a measurable $U\subset\Omega^{n-k}$ is given by
\label{thm:spharray}
\begin{align*}
 \Vol(U) = \int_U \Vol(S^{k-1}(1)) [f(x_{k+1}, \dots, x_n)]^{k-1} \sqrt{1 + f_{x_{k+1}}^2 + \dots + f_{x_n}^2} \mbox{d} x_{k+1} \dots \mbox{d} x_n.
\end{align*}
\end{lemma}
\begin{proof}
We can represent half of a fiber $S^{k-1}( f(x_{k+1}, \dots, x_n) )$ as the graph of the function
$$
 x_1= h(x_2, \dots, x_n) = \sqrt{[f(x_{k+1}, \dots, x_n)]^2 - x_2^2 - \dots - x_k^2}.
$$
By computing
\begin{align*}
h_{x_i} & = \frac{-x_i}{\sqrt{f^2 - x_2^2 - \dots - x_k^2}} \hspace{.25cm} \mbox{ for } \hspace{.25cm} 2 \leq i \leq k, \hspace{.25cm} \mbox{ and} \\
h_{x_i} & = \frac{f f_{x_i}}{\sqrt{f^2 - x_2^2 - \dots - x_k^2}} \hspace{.25cm} \mbox{ for } \hspace{.25cm} k+1 \leq i \leq n,
\end{align*}
we have
\begin{align} \label{sphereintegrand}
1 + h_{x_2}^2 + \dots + h_{x_n}^2 & = \frac{f^2 (1 + f_{x_{k+1}}^2 + \dots + f_{x_n}^2)}{f^2 - x_2^2 - \dots - x_k^2}.
\end{align}

For any measurable $U \subset \Omega^{n - k}$, applying (\ref{sphereintegrand}) in Lemma \ref{lem:hypersurfacegraph} and using Lemma \ref{lem:spherearea} yields
\begin{align*}
\Vol(U) & = 2\int_U \sqrt{1 + h_{x_2}^2 + \dots + h_{x_n}^2} \mbox{ }dx_2 \dots dx_n \\
& =2\int_U \sqrt{ \frac{f^2}{f^2 - x_2^2 - \dots - x_k^2} } \sqrt{1 + f_{x_{k+1}}^2 + \dots + f_{x_n}^2} \mbox{ }dx_2 \dots dx_n \\
& = \int_U \Vol(S^{k-1}(f)) \sqrt{1 + f_{x_{k+1}}^2 + \dots + f_{x_n}^2} \mbox{ }dx_{k+1} \dots dx_n \\
& = \int_U \Vol(S^{k-1}(1)) f^{k-1} \sqrt{1 + f_{x_{k+1}}^2 + \dots + f_{x_n}^2} \mbox{ }dx_{k+1} \dots dx_n,
\end{align*}
where the final equality follows from the fact that $\Vol(S^{k-1}(R)) = R^{k-1} \cdot \Vol(S^{k-1}(1))$.
\end{proof}

\subsection*{Proof of Theorem \ref{Main theorem 1}}

\noindent {\bf Part (a)} The hypotheses of part (a) and Lemma \ref{thm:spharray} imply that the warping function $f$ of $H$ must satisfy the following for every measurable $U \subset \Omega^{n-k}$:
\begin{align}\label{appcondition}
 \int_U \Vol(S^{k-1}(1)) f^{k - 1} \sqrt{1 + f_{x_{k+1}}^2 + \dots + f_{x_n}^2} \mbox{ } dx_{k+1} \dots d x_n & = \Vol(S^{k-1}(1)) \cdot \Vol(U).
\end{align}
It follows that $f$ must be a solution of
\begin{align}
\label{pde1}
[f(x_{k+1}, \dots, x_n)]^{k-1} \cdot \sqrt{1 + f_{x_{k+1}}^2 + \dots + f_{x_n}^2} = 1
\end{align}
on $\Omega^{n-k}$.  Note that the constant function, $f \equiv 1$, is one such solution.  For any solution $f$, $f \leq 1$ on $\Omega^{n-k}$ and on $A = \{ (x_2, \dots, x_n) \in \Omega^{n-k} : 0 < f((x_{k + 1}, \dots, x_n)) < 1 \}$,
\begin{align} \label{processedde}
\left(\frac{f^{k-1}}{\sqrt{1-f^{2k-2}}} \cdot f_{x_{k+1}} \right)^2 + \cdots + \left(\frac{f^{k-1}}{\sqrt{1-f^{2k-2}}} \cdot f_{x_n} \right)^2 = 1.
\end{align}
\noindent The substitution
\begin{align*}
\omega(x_{k+1}, \dots, x_n) = \int_0^{f(x_{k+1},\dots ,x_n)} \frac{t^{k-1}}{\sqrt{1-t^{2k-2}}} \mbox{ } dt
\end{align*}
transforms (\ref{processedde}) into $(\omega_{x_{k+1}})^2 +\cdots + (\omega_{x_n})^2 = 1$, that is, $\| \nabla \omega \| = 1$.
This means that everywhere on $A$, $f = f_k \circ \omega$, where $f_k$ is the {\em Archimedean scaling function} given by
\begin{eqnarray}\label{eo}
f_k^{-1}(y) = \int_0^y \frac{t^{k-1}}{\sqrt{1-t^{2k-2}}} \mbox{ } dt.
\end{eqnarray}
Here $f_k: [0, M_k] \rightarrow [0,1]$, where $M_k = \int_0^1 t^{k-1}/\sqrt{1-t^{2k-2}} \mbox{ } dt$.  The graph of $f_k$ is indicated in Figure \ref{fig:profile};
an explicit formula for $M_k$ is given in Section 4.

\medskip

\noindent {\bf Part (b)} Under the hypotheses of part (b), $f = f_k \circ \omega$ is continuously differentiable on $B$ and satisfies (\ref{pde1}), and therefore (\ref{appcondition}), for all measurable $U \subset B$.  \hfill $\square$

\subsection*{Proof of Theorem \ref{Main theorem 2}}

Suppose that the spherical array $H = \Omega^{n-k} \times_f S^{k-1} \subset \R^{n-k} \times \R^k$ is an Archimedean spherical array, and that the constant of proportionality is $C = \Vol(S^{k-1}(1))$.  Since $H$ is closed, $\Omega^{n-k}$ is bounded and $f = 0$ on $\partial \Omega^{n-k}$ (the boundary of $\Omega^{n-k}$).  Since $H$ is of at least $C^2$ smoothness, $f$ must be continuously differentiable on the interior of $\Omega^{n-k}$.  By part (a) of Theorem \ref{Main theorem 1} (and its proof), $0 < f(x^{\prime \prime}) \leq 1$ for all $x^{\prime \prime}$ in the interior of $\Omega^{n-k}$.  Moreover, since $H$ bounds a strictly convex interior, $\Omega^{n-k}$ is strictly convex and $f(I) = 1$ for at most one point $I$ in the interior of $\Omega^{n-k}$.
Therefore by part (a) of Theorem \ref{Main theorem 1}, everywhere on the interior of $\Omega^{n-k}$ with the possible exception of one point $I$, $f = f_k \circ \omega$ where $\omega$ is a continuously differentiable solution of $\| \nabla \omega(x^{\prime \prime}) \| = 1$.

Since $f = 0$ on $\partial \Omega^{n-k}$, $\omega$ satisfies the boundary condition $\omega = 0$ on $\partial \Omega^{n-k}$, so $\omega$ is the {\em distance to the boundary function}, that is, $\omega (x^{\prime \prime}) = \mbox{\rm dist}(x^{\prime \prime}, \partial \Omega^{n-k})$.  Suppose now that we extend $\omega$ to be the {\em signed distance function} defined on all of $\R^{n-k}$:
\[ \widetilde{\omega}(x^{\prime \prime}) = \begin{cases}
      \mbox{\rm dist}(x^{\prime \prime}, \partial \Omega^{n-k}) & \textrm{if $x^{\prime \prime} \in \Omega^{n-k}$} \\
      -\mbox{\rm dist}(x^{\prime \prime}, \partial \Omega^{n-k})& \textrm{if $x^{\prime \prime} \not\in \Omega^{n-k}$.} \\
   \end{cases} \]
We have already established that $\widetilde{\omega}$ is continuously differentiable everywhere on the interior of $\Omega^{n-k}$ with the possible exception of one point $I$, and the (strict) convexity of $\Omega^{n-k}$ is enough to ensure that $\widetilde{\omega}$ is differentiable everywhere outside of $\Omega^{n-k}$.  Because $H$ is of at least $C^2$ smoothness, $\partial \Omega^{n-k}$ is of at least $C^2$ smoothness, so $\widetilde{\omega}$ is of at least $C^2$ smoothness in some neighborhood of $\partial \Omega^{n-k}$ (see \cite{Krantz}).  Therefore, $\widetilde{\omega}$ satisfies $\| \nabla \widetilde{\omega}(x^{\prime \prime}) \| = 1$, and  $\widetilde{\omega}$ is differentiable on all of $\R^{n-k}$ with the possible exception of one point $I$ in the interior of $\Omega^{n - k}$.

By a result of L.\ A.\ Caffarelli and M.\ G.\ Crandall \cite{Caffarelli}, since the set of singularities of of $\widetilde{\omega}$ has Hausdorff 1-measure zero, $\widetilde{\omega}$ must either be affine (which is certainly not the case) or a ``cone function'':  $\widetilde{\omega}(x^{\prime \prime}) = a + || x^{\prime \prime} - z ||$ for some $a \in \R$ and $z \in \R^{n-k}$.  Thus, $\Omega^{n - k}$ is a ball, and the center of the ball, which is the one point on its interior where $\omega$ fails to be differentiable, is the only point $I$ where $f(I) = 1$.  Since $f(I) = 1$, $\omega(I)$ = $M_k$, so $\Omega^{n-k}$ is an $(n-k)$-dimensional ball of radius $M_k$, and the warping function $f = f_k \circ \omega$, where $\omega$ is the distance to the boundary function associated with this ball.

Therefore, the only spherical array, $H = \Omega^{n-k} \times_f S^{k-1} \subset \R^{n-k} \times \R^k$, that could possibly be an Archimedean spherical array in the sense of Definition \ref{ASA}, with constant of proportionality $C = \Vol(S^{k-1}(1))$, is $\mathscr{A}_k^{n-1} = B^{n-k}(M_k) \times_{f_k \circ \omega} S^{k-1}$, where $\omega$ is the distance to the boundary function on the base $B^{n-k}(M_k)$.  Finally, using the differential equation (\ref{pde1}), it is straightforward to check that $\mathscr{A}_k^{n-1}$ does, in fact, satisfy the APP corresponding to orthogonal projection onto its base with constant of proportionality $C = \Vol(S^{k-1}(1))$.  It is also straightforward to check that $\mathscr{A}_k^{n-1}$ is of at least $C^2$ smoothness, and the proof of Theorem \ref{Main theorem 3} determines the exact degree of smoothness of  $\mathscr{A}_k^{n-1}$.  \hfill $\square$

\medskip

\subsection*{Proof of Theorem \ref{Main theorem 3}}

The warping function $f$ of the Archimedean spherical array $\mathscr{A}_k^{n-1}$ is given by $f = f_k \circ \omega$, where $f_k$ is the Archimedean scaling function defined by (\ref{eo}) and $\omega$ is the distance to the boundary function on its base $B^{n-k}(M_k)$, which we may assume is centered at the origin in $R^{n-k}$ so that
\begin{equation} \label{balldistance}
\omega(x^{\prime \prime}) = M_k-\|x^{\prime \prime}\| = M_k-\sqrt{x_{k+1}^2 + \cdots + x_n^2}.
\end{equation}
There is no useful closed-form expression for $f_k$, but we have the following.
\begin{lemma} \label{smoothnessoffk}
The function $y = f_k(x)$ is analytic on the interval $(0, M_k]$.  In particular, $f_k$ is given by a power
series of the form
\begin{align}
\label{series}
f_k(x) = 1 + \frac{f_k^{\prime \prime}(M_k)}{2!}(x - M_k)^2 + \frac{f_k^{(4)}(M_k)}{4!}(x - M_k)^4 + \cdots
\end{align}
whose radius of convergence is $M_k$.
\end{lemma}
\begin{proof}
The issue is the behavior of $y = f_k(x)$ at the point $(M_k,1)$, where the integral representation
(\ref{eo}) of the inverse function $f_k^{-1}$ is singular.  In \cite{CollandDodd},
it is shown that the function $y = f_k(x)$, there written in the form
\begin{equation} \nonumber
x = f_k^{-1}(y) = \frac{1}{2} \int_0^{y^2} \frac{t^{(k-1)/2}}{\sqrt{1-t^{k-1}}} \mbox{ } dt,
\end{equation}
is a solution of the initial value problem
\begin{eqnarray}\label{DE}
y^{2k-2}+y^{2k-4}(yy^\prime)^2=1, \hspace{.1in} y(0) = M_k.
\end{eqnarray}
An inductive argument based on the differential equation (\ref{DE}) is used to establish
that $y = f_k(x)$ has derivatives of all orders on the interval $(0, M_k]$, and moreover that
for odd $j \geq 1$, $f_k^{(j)}(M_k) = 0$.  In \cite{CollDoddHarrison} this inductive argument is extended to
show that for all $j \geq 1$, $(-1)^jf_k^{(j)}(x) < 0$ for $0 < x < M_k$.  The lemma then follows from
Bernstein's theory of absolutely monotonic and completely monotonic functions (see \cite{Bernstein} and
\cite{widder}).
\end{proof}
Recall now that
\begin{equation} \label{ahdefinition}
\mathscr{A}_k^{n-1} = \left\{ (x_1, \ldots, x_n) \in \R^n : x_1^2 + \cdots + x_k^2 = \left[ f_k \left( \omega(x_{k+1}, \ldots, x_n) \right) \right]^2 \right\}
\end{equation}
where $\omega$ is given by (\ref{balldistance}) and $f_k$ is given by (\ref{series}).  The function $\omega$ is not
differentiable at the origin in $\R^{n-k}$, but the composition with $f_k$ smooths out this singularity, leading to
\begin{equation} \label{analyticexpression}
\mathscr{A}_k^{n-1} = \left\{ (x_1, \ldots, x_n) \in \R^n : x_1^2 + \cdots + x_k^2 = \left[ \sum_{j = 0}^\infty \frac{f^{(2j)}(M_k)}{(2j)!} (x_{k+1}^2 + \cdots + x_n^2)^{\, j} \right]^2 \right\}
\end{equation}
where the series in (\ref{analyticexpression}) converges to an analytic function on any compact $K \subset \R^{n-k}$ in the interior of $B^{n-k}(M_k)$.  For any point $p = (x_1, \ldots, x_n) \in \mathscr{A}_k^{n-1}$  where $\sqrt{x_{k+1}^2 + \cdots + x_n^2} < M_k$, $x_i \neq 0$ for some $i$ between 1 and $k$, so there is a neighborhood of $p$ on which $\mathscr{A}_k^{n-1}$ is the graph of the analytic function obtained by solving (\ref{analyticexpression}) for $x_i$.

It remains to determine the degree of smoothness of $\mathscr{A}_k^{n-1}$ at a point $p$ lying on the boundary of its base $B^{n-k}(M_k)$, that is, $p = (x_1, \ldots, x_n)$ where $x_1 = x_2 = \cdots = x_k = 0$, $\sqrt{x_{k+1}^2 + \cdots + x_n^2} = M_k$ and $\omega(x_{k+1}, \ldots, x_n) = 0$.  Because $f_k(\omega)$ is not differentiable at $\omega = 0$, the right hand side of (\ref{ahdefinition}) is not differentiable at such a point $p$.  However, we can rewrite (\ref{ahdefinition}) this way:
\begin{equation} \label{sidewaysAH}
\mathscr{A}_k^{n-1} = \left\{ (x_1, \ldots, x_n) \in \R^n : f_k^{-1} \left( \sqrt{x_1^2 + \cdots + x_k^2} \right) = M_k - \sqrt{x_{k+1}^2 + \cdots + x_n^2} \right\}.
\end{equation}
At any point $p$ on the boundary of $B^{n-k}(M_k)$, $x_i \neq 0$ for some $i$ between $k + 1$ and $n$, so there is a neighborhood of $p$ on which $\mathscr{A}_k^{n-1}$ is the graph of the function obtained by solving (\ref{sidewaysAH}) for $x_i$.  The degree of smoothness of this graph is precisely that of the function $f_k^{-1} ( \sqrt{x_1^2 + \cdots + x_k^2} )$, which by Lemma \ref{smoothnessoffk} is analytic, except possibly where $x_1 = x_2 = \cdots = x_k = 0$.  Its smoothness is fully resolved by the following, which is Theorem 3.4 of \cite{CollDoddHarrison}.
\begin{lemma}
At $x_1 = x_2 = \cdots = x_k = 0$, the function $f_k^{-1}(\sqrt{x_1^2 + \cdots + x_k^2})$ is an analytic function of the variables $x_1, \dots, x_k$ when $k$ is even and a $C^{k-1}$ function of the variables $x_1,\dots, x_k$ when $k$ is odd. \hfill $\square$
\end{lemma}

\section{The Volumes of the Archimedean Spherical Arrays}
By construction, the largest $(k-1)$-dimensional spherical fiber of $\mathscr{A}_k^{n-1}$, which occurs at the center of the base, has radius $1$.  Let $\mathscr{A}_k^{n-1}(R)$ be the scaled Archimedean spherical array for which the largest spherical fiber has radius $R$.  Then $\mathscr{A}_k^{n-1}(R)$ satisfies the APP corresponding to orthogonal projection onto its base with constant of proportionality, $C = \Vol(S^{k-1}(R))$.  The $(n - 1)$-volume of $\mathscr{A}_k^{n-1}(R)$ is given by the following beautiful formula.
\begin{theorem}
\label{thm:volume}
The $(n-1)$-volume of $\mathscr{A}_k^{n-1}(R)$ is
\begin{align}
\label{AHvol}
{\normalfont \Vol}(\mathscr{A}_k^{n-1}(R)) = \frac{\pi^{\frac{2n-k}{2}}}{2^{n-k-2} \cdot (k-1)^{n-k}} \cdot \frac{\left(\Gamma(\frac{k}{2k-2})\right)^{n-k}}{\Gamma(\frac{n-k}{2})\cdot \Gamma(\frac{k}{2}) \cdot \left(\Gamma(\frac{2k-1}{2k-2})\right)^{n-k}} \cdot R^{n-1}.
\end{align}
\end{theorem}
\begin{proof} We begin with the observation that (\ref{eo}) takes the form of  the {\it incomplete Beta function}:
\begin{align}
\label{betafunction}
B(z;p,q) = \int_0^z u^{p - 1} (1 - u)^{q - 1} \mbox{ } du,
\hspace{.25cm} \mbox{Re}(p) > 0 \mbox{ and } \mbox{Re}(q) > 0.
\end{align}
Indeed, substituting $u = t^{2k-2}$ into (\ref{eo}) and using (\ref{betafunction}) yields
\begin{align*}
f^{-1}_k(y) = \frac{1}{2k-2} B\left(y^{2k-2};\frac{k}{2k-2},\frac{1}{2}\right), \hspace{.25cm} 0 \leq y \leq 1.
\end{align*}
The function $B(p,q) := B(1;p,q)$ is called the \emph{complete Beta function}.   Noting that
\begin{align*}
M_k = f^{-1}_k(1) = \frac{1}{2k-2} B\left(\frac{k}{2k-2},\frac{1}{2}\right),
\end{align*}
and using the identity
\begin{align*}
B(p,q) = \frac{\Gamma(p)\Gamma(q)}{\Gamma(p+q)}
\end{align*}
together with the fact that $\Gamma(1/2) = \sqrt{\pi}$ gives
\begin{align}
\label{xmaxgamma}
M_k = \frac{\sqrt{\pi}}{2k-2} \cdot \frac{\Gamma(\frac{k}{2k-2})}{\Gamma(\frac{2k-1}{2k-2})}.
\end{align}
Now $\Vol(\mathscr{A}_k^{n-1}(R)) = \Vol(\mathscr{A}_k^{n-1}(1)) R^{n-1}$ and by the APP it follows that
\begin{align*}
\Vol(\mathscr{A}_k^{n-1}(1)) = \Vol(S^{k-1}(1)) \cdot \Vol(B^{n-k}(M_k)).
\end{align*}
Using (\ref{xmaxgamma}) along with the well-known volume formulas
\begin{align*}
\Vol(S^{m-1}(\delta))  = \frac{2 \cdot\pi^{m/2}}{\Gamma(m/2)} \cdot \delta^{m-1} \hspace{.75cm} \mbox{and} \hspace{.75cm}
\Vol(B^m(\delta))  = \frac{2 \cdot \pi^{m/2}}{m \cdot \Gamma(m/2)} \cdot \delta^m
\end{align*}
yields (\ref{AHvol}).
\end{proof}
\begin{remark}
When the codimension $k=n-1$, $\mathscr{A}_{n-1}^{n-1}(R)$ is an equizonal ovaloid and equation (\ref{AHvol}) reduces to
\begin{eqnarray}\label{old prop2}
\Vol[\mathscr{A}_{n-1}^{n-1}(R)] = \frac{2\pi^{\frac{n}{2}}}{(n - 2)} \hspace{.05in}
\frac{\Gamma \left( \frac{n-1}{2n - 4} \right) }
{\Gamma \left(\frac{n-1}{2} \right) \Gamma \left( \frac{2n-3}{2n-4} \right) }
\hspace{.05in} R^{n-1}.
\end{eqnarray}
\noindent In \cite{CollandDodd}, it is shown that the $n$-volume of the region bounded by the equizonal ovaloid $\mathscr{A}_{n-1}^{n-1}(R)$ is given by another nice formula,
\begin{eqnarray}\label{old prop3}
\frac{2\pi^{\frac{n}{2}}}{(n-1)(n - 2)}
\frac{\Gamma \left(\frac{n-1}{n-2} \right)}
     {\Gamma \left( \frac{n-1}{2} \right) \Gamma \left( \frac{3n-4}{2n-4} \right) } R^{n}.
\end{eqnarray}
Of course when $n = 3$, the formulas in (\ref{old prop2}) and (\ref{old prop3})
reduce to $4 \pi R^2$ and $4 \pi R^3/3$, respectively.  As yet, we have no formula for the $n$-volume of the region bounded by the general Archimedean spherical array $\mathscr{A}_{k}^{n-1}(R)$.
\end{remark}

\section{Comments and Questions}
An Archimedian spherical array is designed to satisfy the APP corresponding to orthogonal projection in one particular direction:  onto its base.  Other than spheres, do there exist {\em ovaloids} (by which we mean closed hypersurfaces embedded in $\R^n$ of at least $C^2$ smoothness and bounding convex interiors) that satisfy APPs corresponding to more than one projection direction?  There are two ways to ask this question:  in the {\em strong sense}, requiring the proportionality constant to be the same for each projection direction, or in the {\em weak sense}, allowing the constant of proportionality to vary with the projection direction.

A classical result of W.\ Blaschke \cite{Blaschke}, stated in our language, is that if an ovaloid $H$ embedded in $\R^3$ satisfies the APP corresponding to every codimension $2$ orthogonal projection in the strong sense, then $H$ is a $2$-sphere.  O.\ Stamm \cite{Stamm} improved this result by showing that it is still true under the weaker hypothesis that $H$ satisfies the APP corresponding to every codimension $2$ orthogonal projection in the weak sense.

A generalization of Blaschke's original result, although not explicitly stated, is implicit in a recent paper of D. \-S.\ Kim and Y.\ H. Kim \cite{Kim}.   Their results imply that if an ovaloid $H$ embedded in $\R^n$ satisfies the APP corresponding to every codimension $n - 1$ orthogonal projection in the strong sense, then $H$ is an $(n - 1)$-sphere.  But, as shown in detail by Rudin \cite{Rudin}, an $(n - 1)$-sphere satisfies the codimension $k$ APP only for $k = 2$.  Thus, if an ovaloid $H$ embedded in $\R^n$ satisfies the APP corresponding to every codimension $n - 1$ orthogonal projection in the strong sense, then $H$ is an $(n - 1)$-sphere and $n - 1 = 2$, i.e., $H$ is a $2$-sphere embedded in $\R^3$.

We can also rule out the existence of non-spherical ovaloids in $\R^n$ satisfying the APP corresponding to every codimension $k$ orthogonal projection in the strong sense for two different values of $k$. If an ovaloid $H$ embedded in $\R^n$ satisfies the APP corresponding to every codimension $k$ orthogonal projection in the strong sense, then the {\em $k$-th projection function for $H$}, that assigns to each $k$-dimensional subspace $S$ of $\R^n$ the $k$-dimensional volume of the orthogonal projection of $H$ onto $S$, is constant.  But it follows from recent results of D.\ Hug \cite{Hug} that if such an ovaloid $H$ has constant $i$-th and $j$-th projection functions for any integers $i$ and $j$ such that $1 \leq i < j \leq n - 2$, with $(i,j) \neq (1, n - 2)$, then $H$ must be an $(n - 1)$-sphere.

This leaves many open avenues for investigation.  For example, does there exist a non-spherical ovaloid $H$ embedded in $\R^n$ satisfying the APP corresponding to every codimension $n - 1$ orthogonal projection in the weak sense?  More generally, for a single projection codimension $k$, $2 \leq k \leq n - 2$, does there exist a non-spherical ovaloid $H$ embedded in $\R^n$ satisfying the APP corresponding to multiple codimension $k$ orthogonal projections, perhaps even every codimension $k$ orthogonal projection, in either the weak sense or the strong sense?  Can a non-spherical ovaloid $H$ embedded in $\R^n$ satisfy the APP corresponding to every codimension $k$ projection in the weak sense for two different values of $k$?

Finally, we note that the equizonal ovaloids $\mathscr{A}_{n-1}^{n-1}$ are also characterized by a remarkable curvature condition \cite{CollHarrison}.  Namely, the principal curvature in the axial direction is the constant multiple $n-1$ of the common value of the $n-1$ principal curvatures in the rotational directions.  This links the unique $n$-dimensional unit equizonal ovaloid with the unique $n$-dimensional minimal hypersurface of revolution, for which the principal curvature in the axial direction is $-(n-1)$ times the shared value of the other principal curvatures. In particular, this establishes a connection between the sphere and the catenoid in the case when $n=2$.  More generally, we observe that there is duality between the codimension $n - 1$ APP of the equizonal ovaloids (which are compact) and the minimality property of the generalized catenoids \cite{Pinl} (which are non-compact).  Is there a relationship between Archimedean hypersurfaces and the minimal spherical arrays of W.\ Hsiang, Z.\ Teng, and W.\ Yu \cite{hsiang}?  Is the APP, in some sense, dual to minimality?

\smallskip

\noindent \textbf{Acknowledgements}  The authors thank David L.\ Johnson, Joe Yukich, and Rob Neel for reading and commenting on various drafts of this paper, and Jerry King for bringing Bernstein's theory to our attention.

\footnotesize

\medskip

\noindent Vincent Coll\\
Department of Mathematics\\
Lehigh University\\
Bethlehem, PA  18015 (USA)\\
e-mail vec208@lehigh.edu

\bigskip

\noindent Jeff Dodd\\
Mathematical, Computing, and Information Sciences Department\\
Jacksonville State University\\
Jacksonville, AL  36265 (USA) \\
e-mail jdodd@jsu.edu

\bigskip

\noindent Michael Harrison\\
Department of Mathematics\\
Pennsylvania State University\\
University Park, PA  16802 (USA) \\
e-mail mah5044@gmail.com


\begin{thebibliography}{abcd}

\bibitem{Bernstein}
S.\ Bernstein, \emph{Sur la d\'{e}finition et les propri\'{e}t\'{e}s des fonctions analytiques d'une variable r\'{e}elle}. Mathematische Annalen \textbf{75}, 449-468 (1914).

\bibitem{Bezdek}
K.\ Bezdek, R.\ Connelly, \emph{Pushing disks apart - the Kneser-Poulsen conjecture in the plane}, Journal f\"ur die reine und angewandte Mathematik {\bf 53}, 221--236 (2001).

\bibitem{Blaschke}
W.\ Blaschke, \emph{Vorlesungen \"uber Differentialgeometrie I}, Springer, Berlin (1921).

\bibitem{Caffarelli}
L.\ A.\ Caffarelli and M.~G.~Crandall, \emph{Distance functions and almost global solutions of eikonal equations},  Communications in Partial Differential Equations {\bf 35}, 391--414 (2010).

\bibitem{CollDoddHarrison}
V.\ Coll, J.\ Dodd, and M.\ Harrison, \emph{On the smoothness of the equizonal ovaloids},
Journal~of~Geometry \textbf{103}, 409--416 (2012).

\bibitem{CollHarrison}
V.\ Coll and M.\ Harrison, \emph{Hypersurfaces of revolution with proportional principal curvatures},
Advances~in~Geometry \textbf{13}, 485--496 (2013).

\bibitem{CollHarrison2}
V.\ Coll and M.\ Harrison, \emph{Gabriel's horn-a revolutionary tale},
Mathematics Magazine \textbf{87}, 263--274 (2014).

\bibitem{CollandDodd}
J.\ Dodd and V.\ Coll, \emph{Generalizing the equal area zones property of the sphere},
Journal~of~Geometry \textbf{90}, 47--55 (2008).

\bibitem{hsiang}
W.\ Hsiang, Z.\ Teng, W.\ Yu, \emph{New examples of constant mean curvature immersions of $(2k - 1)$-spheres into Euclidean $2k$-space}, The Annals of Mathematics, Second Series {\bf 117}(3), 609--625 (1983).

\bibitem{Hug}
D.\ Hug, \emph{Nakajima's problem for general convex bodies}, Proceedings of the American Mathematical Society \textbf{137}(1), 255--263 (2009).

\bibitem{Krantz}
S.\ Krantz and H.\ Parks, \emph{Distance to $C^k$ Hypersurfaces}, Journal of Differential Equations {\bf 40}, 116--120 (1981).

\bibitem{Kim}
D.\ S.\ Kim and Y.\ H. Kim, \emph{Some characterizations of spheres and elliptic paraboloids}, Linear Algebra and its Applications {\bf 437}, 113--120 (2012).

\bibitem{Li}
Y.\ Li and L.\ Nirenberg, \emph{The distance function to the boundary, Finsler geometry, and the singular set of viscosity solutions of some Hamilton-Jacobi equations}, Communications on Pure and Applied Mathematics {\bf 58}, 85--146 (2005).

\bibitem{Pinl}
M.\ Pinl and W.\ Ziller, \emph{Minimal hypersurfaces in spaces of constant curvature}, Journal of Differential Geometry \textbf{11}, 335–-343 (1976).

\bibitem{Rudin}
W.\ Rudin, \emph{A generalization of a theorem of Archimedes}, The American Mathematical Monthly {\bf 80}(7), 794--796 (1973).

\bibitem{Stamm}
O.\ Stamm, \emph{Umkehrung eines Satzes von Archimedes \"uber die Kugel}, Abh.~Math.~Sem.~Univ.~Hamburg {\bf 17}, 112--132 (1951); reviewed in MathSciNet:  MR0041467.

\bibitem{widder}
D.\ Widder, \emph{The Laplace Transform}, Princeton University Press (1941).

\end{thebibliography}
\end{document}